\documentclass[11pt, letterpaper]{amsart}
\usepackage[left=1in,right=1in,bottom=1.2in,top=1in]{geometry}
\usepackage{amsfonts}
\usepackage{amsmath, amssymb}
\usepackage{graphicx}
\usepackage[font=small,labelfont=bf]{caption}
\usepackage{epstopdf}
\usepackage[pdfpagelabels,hyperindex]{hyperref}
\usepackage{xcolor}
\usepackage{amsthm}
\usepackage{float}
\usepackage{pgfplots}
\usepackage{listings}
\usepackage{longtable}
\usepackage{mathrsfs}

\DeclareGraphicsExtensions{.pdf,.jpg,.png,.ps,.eps,.gif}

\theoremstyle{plain}
\newtheorem{theorem}{Theorem}\numberwithin{theorem}{section}
{}
{}
\newtheorem{lemma}{Lemma}\numberwithin{lemma}{section}
\newtheorem{proposition}{Proposition}\numberwithin{proposition}{section}
\numberwithin{proposition}{section}
\numberwithin{corollary}{section}
\theoremstyle{definition}
\numberwithin{definition}{section}
\theoremstyle{remark} \theoremstyle{exam} \theoremstyle{ob}
\numberwithin{remark}{section}
\numberwithin{question}{section}
\numberwithin{exam}{section}
\numberwithin{ob}{section}
\numberwithin{nt}{section}
\numberwithin{equation}{section}

\begin{document}
\title[Weak type 1-1 bound of multi-parameter maximal function]
{Weak type 1-1 bound of multi-parameter maximal function}

\author
{Hoyoung Song}

\address{Hoyoung Song, Department of Mathematics \\
	Yonsei University \\
	Seoul 120-729, Korea}
\email{nonspin0070@gmail.com}


 \begin{abstract}
We define  the mulati-parameter maximal function $\mathcal{M}$ as
$$
\mathcal{M} f(x)=\sup _{0<h_1,h_2,\cdots,h_n<1} \frac{1}{h_1h_2\cdots h_n}\left|\int_0^{h_1}\cdots \int_0^{h_n} f(x-P(t_1,\cdots,t_n)) \mathrm{d}t_1\cdots \mathrm{d} t_n\right|
$$
where $P(t_1,t_2,\cdots,t_n)$ is a real-valued multi-parameter polynomial of real variables $t_1,t_2,\cdots,t_n$. Then, we prove that $\mathcal{M}$ is of weak-type 1-1 with a bound that depends only on the coefficients of $P(t_1,t_2,\cdots,t_n)$.
\end{abstract}
\maketitle

\setcounter{tocdepth}{1}

\tableofcontents

\section{Introduction}
If $T$ is a mapping from $L^1\left(\mathbb{R}^n\right)$ to $L^{1, \infty}\left(\mathbb{R}^n\right)$, then we say that $T$ is a weaktype 1-1 operator if
$$
|\{x: T f(x)>\alpha\}| \leq \frac{C\|f\|_{L^1}}{\alpha}
$$where $C$ is a constant independent of  $f$ and $\alpha.$  In [1], Carbery, Ricci and Wright obtained the uniform weak-type 1-1 bounds of
$$
\begin{aligned}
	& M_0 f(x)=\sup _{0<h<1} \frac{1}{h} \int_0^h|f(x-P(s))| \mathrm{d} s, \\
	& H_0 f(x)=p \cdot v . \int_{|s|<1} f(x-P(s)) \frac{\mathrm{d} s}{s}
\end{aligned}
$$ where $P$ is real-valued polynomial.
In the same spirit of the one-parameter maximal operator, we shall consider the multi-parameter and one-dimensional operators
		$$
M f(x)=\sup _{0<h_1,h_2<1} \frac{1}{h_1h_2}\left|\int_0^{h_1}\int_0^{h_2} f(x-P(t_1,t_2)) \mathrm{d}t_1\mathrm{d} t_2\right|,
	$$
	$$
\mathcal{M} f(x)=\sup _{0<h_1,h_2,\cdots,h_n<1} \frac{1}{h_1h_2\cdots h_n}\left|\int_0^{h_1}\cdots \int_0^{h_n} f(x-P(t_1,\cdots,t_n)) \mathrm{d}t_1\cdots \mathrm{d} t_n\right|.
$$
The $L^p(\mathbb{R}), p>1$ boundedness of $\mathcal{M}$ is obtained by Ricci and Stein's Theorem 7.1 of \cite{7} with de Leeuw's theorem. In \cite{P}, Patel showed that $\mathfrak{M}$ is of weak-type 1-1 with a bound that depends only on the coefficients of $P$. There are also other multi-parameter sigular integral studies. In \cite{2}, Carbery, Wainger and Wright obtained the necessary and sufficient condition for $H$ to be bounded on $L^{p}$, $1 <p< \infty$ using the Newton diagram of $P$ where $$H f(x, y, z)=p \cdot v \cdot \int_{|s|<1} \int_{|t|<1} f(x-s, y-t, z-P(s, t)) \frac{\mathrm{d} s \mathrm{d} t}{s t}.$$ 
In \cite{6}, the necessary and sufficient condition for the $L^{p}$ boundedness of global version of $H$ was obtained by Patel.  More recently,  the  multiple Hilbert transforms were studied  in \cite{CWW2}, \cite{CHKY} and 
\cite{K} where  $$\mathcal{H} f(x_1,x_2,\cdots, x_n)= \int f(x_1-P_1(t_1,\cdots,t_m), \cdots, x_n-P_n(t_1,\cdots,t_m)) \frac{\mathrm{d} t_1\cdots \mathrm{d} t_m}{t_1\cdots t_m}.$$ 
Unlike the $L^{p}$ boundedness, we do not know much about the weak-type 1-1 estimate. We define $M_1 f(x, z)$ and $H_1 f(x, z)$ as
$$
\begin{aligned}
	& M_1 f(x, z)=\sup _{0<h<1} \frac{1}{h} \int_0^h|f(x-s, z-P(s))| \mathrm{d} s, \\
	& H_1 f(x, z)=p . v \cdot \int_{|s|<1} f(x-s, z-P(s)) \frac{\mathrm{d} s}{s} .
\end{aligned}
$$
Although we have $L^{p}$ boundedness of $M_1$ and $H_1$ (independent of the coefficients of $P$) for $1<p<\infty$, weak-type 1-1 bounds are not known for $M_1$ and $H_1$. We know the necessary and sufficient condition for $L^{p}$ boundedness of $H_2$\cite{P1} where 
$$H_2 f(x)=p \cdot v \cdot \int_{|s|<1} \int_{|t|<1} f(x-P(s, t)) \frac{\mathrm{d}s \mathrm{d}t}{s t}.$$ However, there is no result about the exact necessary and sufficient condition for  weak-type 1-1 boundedness of $H_2$.
In this paper, we obtain the weak-type 1-1 bound of $\mathcal{M}$ with a bound depending only on the coefficients of $P$.

\begin{theorem}\label{5a}
	 $\mathcal{M}$ is of weak-type 1-1 with a bound that is dependent only on the coefficients of $P$.
\end{theorem}

\subsection{Organization}
In Section 2, we introduce the Newton diagram of $P$ and its properties and treat the monomial case. In Section 3, we decompose $\mathcal{M}f(x)$ with vertices of the Newton diagram of $P$. Then, we treat the non-zero coordinate case and zero coordinate case of the vertex separately. For the two cases, we introduce two important  Lemmas and show the weak-type 1-1 boundedness of $\mathcal{M}f(x)$ by the Calderón-Zygmund singular integral theory combined with the two Lemmas. In Section 4, we prove one of the two Lemmas for the non-zero coordinate case by the properties of Newton diagram and some techniques of Harmonic analysis.  In Section 5, we also show the other Lemma for the zero coordinate case. For this, we apply the sublevel set estimate and some estimates of Harmonic analysis with the properties of Newton diagram.

\section{Preliminaries}
Let $\mathfrak{t}=(t_1,t_2,\cdots,t_n)$ and $\mathfrak{m}=(m_1,m_2,\cdots,m_n)$. Then, we can rewrite the polynomial $P(t_1,\cdots,t_n)$ as $P(\mathfrak{t})=\sum_{\mathfrak{m} \in \Lambda} a_{\mathfrak{m}} \mathfrak{t}^{\mathfrak{m}}$, where $\Lambda$ is indexing the set of lattice points in $\mathbb{Z}^n$ such that $a_{\mathfrak{m}} \neq 0$. Let $\mathbb{N}_{\star}$ denote the set of non-negative integers.  For the set $A=\{a_1,a_2\cdots,a_n\}\subset\{1,2\cdots,n\}$, we set $\int f(t)dt_{i\in A}=\int f(t)dt_{a_1}dt_{a_2}\cdots dt_{a_n}$. $A \sim B$ will be used to indicate that $A$ and $B$ are essentially similar in size, unimportant error terms are neglected. For the vectors $\bar{v}$ and $\bar{w}$, let $\bar{v} \ge \bar{w}$ denote all the components of $\bar{v}$ are greater than or equal to all the components of $\bar{w}$.  
\subsection{Newton Diagram}We shall introduce the Newton Diagram. For the details, see \cite{2}, \cite{K} and \cite{P}. For each $\mathfrak{m} \in \Lambda$, set
$$
Q_{\mathfrak{m}}=\left\{(x_1,x_2,\cdots,x_n) \in \mathbb{R}^n: x_1 \geq m_1,  x_2 \geq m_2, \cdots, x_n \geq m_n\right\} .
$$
Then the smallest closed convex set containing $Q=\bigcup_{(m, n) \in \Lambda} Q_{m, n}$ is called the Newton diagram of $P$. We denote it by $\Pi$ and it is an unbounded polygon with a finite number of corners. Let $\mathcal{D}$ denote the set of its corner points. Furthermore, suppose $\mathcal{D}$ consists of $r$ corner points $v_1, v_2, \ldots, v_r$ with $v_j=(m_{1}^{j},\cdots, m_{n}^{j})$ for $1 \leq j \leq r$.  From now, let non-negative vector denote the vector whoose  all components are not negative number. By the elementary geometry, one knows that for a corner point $v_j$ of $\Pi$, there are exactly $n$'s faces of $\Pi$ intersecting $v_j$. Hence, we can choose $n$'s non-negative vectors $\bar{n}_{j}^{(1)},\bar{n}_{j}^{(2)},\cdots,  \bar{n}_{j}^{(n)}$ that are normal vectors to the faces intersecting on the corner point $v_j$.  By the convexity of the Newton diagram, one has \begin{align}\label{1a}
	 \bar{n}_{j}^{(1)}\cdot\left(v-v_j\right) \geq 0,  \bar{n}_{j}^{(2)}\cdot\left(v-v_j\right) \geq 0 \cdots,  \bar{n}_{j}^{(n)} \cdot\left(v-v_j\right) \geq 0
\end{align} for all $v \in \Lambda$ and $1 \leq j \leq r$. Now for $1 \leq j \leq r$, define \begin{align*}
T(j) & =\left\{(q_1,\cdots,q_n) \in \mathbb{N}^{n} :(q_1,\cdots,q_n) \cdot\left(v-v_j\right)>0 \text { for all } v \in \Lambda \backslash\left\{v_j\right\}\right\} \\
& =\bigcap_{v \in \Lambda \backslash\left\{v_j\right\}}\left\{(q_1,\cdots,q_n) \in \mathbb{N}^{n} :(q_1,\cdots,q_n) \cdot\left(v-v_j\right)>0\right\}.
\end{align*}
Then, we have the following Lemma.  We omit its proof. For the details, see \cite{2}, \cite{K} and \cite{P}.
\begin{lemma}\label{2a}
For $1 \leq j \leq r$,\\
(i)	\begin{align*}
		T(j)=\left\{(q_1,\cdots,q_n) \in \mathbb{N}^{n}:(q_1,\cdots,q_n)=\alpha_1 \bar{n}_{j}^{(1)}+\cdots+\alpha_n \bar{n}_{j}^{(n)} \text{for some positive reals} \ \alpha_1 \cdots, \alpha_2\right\}.
	\end{align*}
(ii) For $1 \leq j<k \leq r$,
$$
T(j) \cap T(k)=\varnothing .
$$
\end{lemma}
Now, we define $S(j)$ as  $$S(j):=\left\{(q_1,\cdots,q_n) \in \mathbb{N}_{\star}^{n}: \exists_{\alpha_1 \geq 0,\cdots,\alpha_n \geq 0}(q_1,\cdots,q_n)=\alpha_1 \bar{n}_{j}^{(1)}+\cdots+\alpha_n \bar{n}_{j}^{(n)}\right\}.$$ Then, 
by the definition of $S(j)$, one can obtain that
	$$
	\bigcup_{j=1}^r S(j)=\mathbb{N}_{\star}^{n}.
	$$
Moreover, with the elemetary calculation, one can know that there exists $d_j>0$ such that
	$$S(j):=\left\{(q_1,\cdots,q_n) \in \mathbb{N}_{\star}^{n}: \exists_{(n_1,\cdots,n_n)\in \mathbb{N}_{\star}^{n}}(q_1,\cdots,q_n)=\frac{n_1}{d_j} \bar{n}_{j}^{(1)}+\cdots+\frac{n_n}{d_j} \bar{n}_{j}^{(n)}\right\}.$$
Let $\mathfrak{q}:=(q_1,\cdots,q_n)$. For  $1< m< n$, we also set $\bar{k}_m:=(k_1,\cdots k_{m-1},k_{m+1},\cdots, k_n)$. Especially, let $\bar{k}_{1}:=(k_2,\cdots, k_n)$ and $\bar{k}_{n}:=(k_1,\cdots, k_{n-1})$. We define $S_1(j),\cdots,S_n(j)$ as for $1\le m \le n$,
\begin{align*}
	S_m(j):=\{\mathfrak{q} \in S(j):\mathfrak{q}=\frac{N+k_{1}}{d_j}\bar{n}_{j}^{(1)}+\cdots+\frac{N+k_{m-1}}{d_j}&\bar{n}_{j}^{(m-1)}+\frac{N}{d_j}\bar{n}_{j}^{(m)}+\frac{N+k_{m+1}}{d_j}\bar{n}_{j}^{(m+1)}+\\
	&\cdots+\frac{N+k_{n}}{d_j}\bar{n}_{j}^{(n)}, \bar{k}_m \in \mathbb{N}_{\star}^{n-1}, N \in \mathbb{N}_{\star}\}.
\end{align*}
Note that $$S(j)=\bigcup_{m=1}^{n}S_m(j).$$
Now, we decompose
$$
S_m(j)=\bigcup_{N \in \mathbb{N}_{\star}} S_m^N(j)
$$
where

\begin{align*}
S_m^N(j):=\{\mathfrak{q} \in S(j):\mathfrak{q}=\frac{N+k_{1}}{d_j}\bar{n}_{j}^{(1)}+\cdots+\frac{N+k_{m-1}}{d_j}\bar{n}_{j}^{(m-1)}&+\frac{N}{d_j}\bar{n}_{j}^{(m)}+\frac{N+k_{m+1}}{d_j}\bar{n}_{j}^{(m+1)}+\\
&\cdots+\frac{N+k_{n}}{d_j}\bar{n}_{j}^{(n)}, \bar{k}_m \in \mathbb{N}_{\star}^{n-1}\}.
\end{align*}

\begin{lemma}\label{1b}For each $1\le j\le r$, there exists $\beta_j>0$ such that
	\begin{align}\label{3a}
		\mathfrak{q} \cdot\left(v-v_j\right) \ge \beta_j N
	\end{align}
	for every $v \in \Lambda \backslash\left\{v_j\right\}$ and $\mathfrak{q} \in S_m^N(j)$.
\end{lemma}\begin{proof}
For every $\mathfrak{q} \in S_m^N(j)$ we can write
$$
\mathfrak{q}=\frac{N+k_{1}}{d_j}\bar{n}_{j}^{(1)}+\cdots+\frac{N+k_{m-1}}{d_j}\bar{n}_{j}^{(m-1)}+\frac{N}{d_j}\bar{n}_{j}^{(m)}+\frac{N+k_{m+1}}{d_j}\bar{n}_{j}^{(m+1)}+\cdots+\frac{N+k_{n}}{d_j}\bar{n}_{j}^{(n)}
$$
for some $\bar{k}_m \in \mathbb{N}_{\star}^{n-1}$.  By the fact that $\bar{n}_{j}^{(1)},\cdots,\bar{n}_{j}^{(n)}$ are linearly independent with $(\ref{1a})$, we have
$$
\left(v-v_j\right) \cdot(\bar{n}_{j}^{(1)}+\cdots+\bar{n}_{j}^{(n)})>0
$$
for all $v \in \Lambda \backslash\left\{v_j\right\}$.  Taking
\begin{align}\label{2b}
	\beta_j:=\min _{v \in \Lambda \backslash\left\{v_j\right\}} \frac{1}{d_j}\left(v-v_j\right) \cdot(\bar{n}_{j}^{(1)}+\cdots+\bar{n}_{j}^{(n)})>0
\end{align}
immediately yields (\ref{3a}).
\end{proof}
\subsection{The Monomial Case} Note that $$\mathcal{M} f(x) \sim \sup _{\mathfrak{q} \in \mathbb{N}_{\star}^{n}} 2^{q_1+\cdots q_n}\left|\int_{2^{-q_1}}^{2^{-q_1+1}} \cdots \int_{2^{-q_n}}^{2^{-q_n+1}} f(x-P(\mathfrak{t})) \mathrm{d}t_1 \cdots\mathrm{d} t_n\right|.$$
\begin{lemma}\label{7a}
If $P(\mathfrak{t})$ is a monomial, then
	$$
	\mathcal{M} f(x) \leq 2 M_H f(x)
	$$
	where $M_H$ denotes the usual Hardy-Littlewood maximal function.
\end{lemma}\begin{proof}
Let $P(\mathfrak{t})=a_{\mathfrak{m}} \mathfrak{t}^{\mathfrak{m}}$ and suppose  that the coefficients of $P$ are the positive real number. Then, after a change of variable $\left(a_{\mathfrak{m}} \mathfrak{t}^{\mathfrak{m}}=u\right)$, we have
 $$
 \begin{aligned}
 	&\mathcal{M} f(x)\sim \sup _{\mathfrak{q}\in \mathbb{Z}_{+}^{n}} \frac{1}{m_n\left(a_{\mathfrak{m}}\right)^{1 / m_n}} 2^{q_1+\cdots+q_{n-1}}\\
 	&\times \int_{t_1=2^{-q_1}}^{2^{-q_1+1}}\cdots \int_{t_{n-1}=2^{-q_{n-1}}}^{2^{-q_{n-1}+1}} \frac{2^{q_n}}{t_1^{m_1 / m_n}\cdots t_{n-1}^{m_{n-1}/m_n}} \int_{a_{\mathfrak{m}} t_1^{m_1}\cdots t_{n-1}^{m_{n-1}} 2^{-m_{n} q_{n}}}^{a_{\mathfrak{m}} t_1^{m_1}\cdots t_{n-1}^{m_{n-1}} 2^{-m_{n} q_{n}+m_{n}}} f(x-u) \frac{\mathrm{d} u}{u^{1-1 / m_n}} dt_1\cdots dt_{n-1}.
 \end{aligned}
 $$
So, it suffices to show that
 \begin{align}\label{4a}
 	\int_{0}^{a_{\mathfrak{m}} t_1^{m_1}\cdots t_{n-1}^{m_{n-1}} 2^{-m_{n} q_{n}+m_{n}}} f(x-u) \frac{\mathrm{d} u}{u^{1-1 / m_n}}  \leq 2 m_n\left(a_{\mathfrak{m}}\right)^{1 / m_n} 2^{-q_n} t_1^{m_1 / m_n}\cdots t_{n-1}^{m_{n-1}/m_n} M_H f(x).
 \end{align}
  This is trivial when $n=1$. When $n>1$, note that
 $$
 \int_0^{a_{\mathfrak{m}} t_1^{m_1}\cdots t_{n-1}^{m_{n-1}} 2^{-m_{n} q_{n}+q_{n}}} f(x-u) \frac{\mathrm{d} u}{u^{1-1 / n}}=\int_{u=0}^{a_{\mathfrak{m}} t_1^{m_1}\cdots t_{n-1}^{m_{n-1}} 2^{-m_{n} q_{n}+q_{n}}} f(x-u)\left(\int_{z=0}^{1 / u^{1-1 / n}} \mathrm{~d} z\right) \mathrm{d} u.
 $$ By changing the order of integral, one can have (\ref{4a}).
\end{proof}

\section{Decomposition And  The Calderón-Zygmund Singular Integral Theory}
Our proof is based on the arguments of \cite{1} and \cite{P}. We define $C^{\infty}$-function $\eta(s)$ supported in $\left[\frac{1}{2}, 4\right]$ as $\eta(s)=1$  if $s\in[1,2]$. Then
$$
\begin{aligned}
	\mathcal{M} f(x) & \lesssim \sup _{\mathfrak{q} \in \mathbb{N}_{\star}^{n}} 2^{q_1+\cdots q_n}\left|\int f(x-P(\mathfrak{t}))\eta(2^{q_1}t_1)\cdots \eta(2^{q_n}t_n) \mathrm{d}t_1 \cdots\mathrm{d} t_n\right| \\
	& =\sup _{\mathfrak{q} \in \mathbb{N}_{\star}^{n}} \left|\int f(x-P(2^{-q_1}t_1,\cdots, 2^{-q_n}t_n))\eta(t_1)\cdots \eta(t_n) \mathrm{d}t_1 \cdots\mathrm{d} t_n\right| \\
	& \leq \sum_{j=1}^r \sup _{\mathfrak{q} \in S(j)}\left|\int f(x-P(2^{-q_1}t_1,\cdots, 2^{-q_n}t_n))\eta(t_1)\cdots \eta(t_n) \mathrm{d}t_1 \cdots\mathrm{d} t_n\right| \\
	& :=\sum_{j=1}^r \mathcal{M}(j) f(x) .
\end{aligned}
$$ Hence, the following Theorem \ref{6a} implies Theorem \ref{5a}
\begin{theorem}\label{6a}
	$\mathcal{M}(j)$ is of weak-type 1-1  for $1 \leq j \leq r$.
\end{theorem}
From now, we shall show Theorem \ref{6a}. For this, we consider two cases (non-zero coordinate case or zero coordinate case).
\subsection{Non-Zero Coordinate Case}
First, we assume that $v_j=(m_{1}^{j},\cdots, m_{n}^{j})$ does not have a vanishing  coordinate. This means that $m_{i}^{j}\neq0$ for all $1\le i \le n$. We can decompose $\mathcal{M}(j) f(x)$ as
\begin{align*}
	& \mathcal{M}(j) f(x)= \sup _{\mathfrak{q} \in S(j)}\left|\int f(x-P(2^{-q_1}t_1,\cdots, 2^{-q_n}t_n))\eta(t_1)\cdots \eta(t_n) \mathrm{d}t_1 \cdots\mathrm{d} t_n\right| \\
	& \leq \sup _{\mathfrak{q} \in S(j)} \left| \int\left[f(x-P(2^{-q_1}t_1,\cdots, 2^{-q_n}t_n))-f\left(x-2^{-\mathfrak{q} \cdot v_j} a_{v_j} \mathfrak{t}^{v_j}\right)\right]\eta(t_1)\cdots \eta(t_n) d\mathfrak{t} \right| \\
	&+\sup _{\mathfrak{q} \in S(j)}\left|\int f\left(x-2^{-\mathfrak{q} \cdot v_j} a_{v_j} \mathfrak{t}^{v_j}\right) \eta(t_1)\cdots \eta(t_n)d\mathfrak{t}\right| \\
	&:=M(j) f(x)+Q(j) f(x) .
\end{align*}
One can apply Lemma \ref{7a} to $Q(j) f(x) $. So, it suffices to consider $M(j) f(x)$. We split $M(j) f(x)$ again as
\begin{align*}
	M(j) f(x)&\le\sum_{m=1}^{n} \sup _{\mathfrak{q} \in S_m(j)} \left| \int\left[f(x-P(2^{-q_1}t_1,\cdots, 2^{-q_n}t_n))-f\left(x-2^{-\mathfrak{q} \cdot v_j} a_{v_j} \mathfrak{t}^{v_j}\right)\right] \eta(t_1)\cdots \eta(t_n) d\mathfrak{t} \right|\\
	&:=\sum_{m=1}^{n}M_{m}(j) f(x).
\end{align*} We can write $P(2^{-q_1}t_1,\cdots, 2^{-q_n}t_n)=2^{-\mathfrak{q}\cdot v_j}\tilde{P}_j(t_1,\cdots,t_n)$ where
$
\tilde{P}_j(\mathfrak{t}):=\sum_{v=\mathfrak{m} \in \Lambda} 2^{-\mathfrak{q} \cdot\left(v-v_j\right)} a_{\mathfrak{m}} \mathfrak{t}^{\mathfrak{m}} .
$ Then, we have \begin{align*}
	M_m(j) f(x)= & \sup _{\mathfrak{q} \in S_m(j)} \mid \int\left[f(x-2^{-\mathfrak{q} \cdot v_j} \tilde{P}_j(\mathfrak{t}))-f(x-2^{-\mathfrak{q} \cdot v_j} a_{v_{j}} \mathfrak{t}^{v_j})\right] \eta(t_1)\cdots \eta(t_n) d\mathfrak{t}t \mid \\
	\leq & \sum_{N \geq 0} \sup _{\mathfrak{q} \in S_m^N(j)} \left| \int\left[f(x-2^{-\mathfrak{q} \cdot v_j} \tilde{P}_j(\mathfrak{t}))-f(x-2^{-\mathfrak{q} \cdot v_j} a_{v_{j}} \mathfrak{t}^{v_j})\right]\eta(t_1)\cdots \eta(t_n) d\mathfrak{t} \right|\\
	= & \sum_{N \geq 0} \sup _{\bar{k}_m \in \mathbb{N}_{\star}^{n-1}} \left| \int\left[f(x-c_N2^{-\sigma_m\cdot \bar{k}_{m}} \tilde{P}_j(\mathfrak{t}))-f(x-c_N2^{-\sigma_m\cdot \bar{k}_{m}} a_{v_{j}} \mathfrak{t}^{v_j})\right]\eta(t_1)\cdots \eta(t_n) d\mathfrak{t} \right|\\
	:= & \sum_{N \geq 0} M_j^N f(x)
\end{align*}
where $c_N:=2^{-(N/d_j)[(\bar{n}_{j}^{(1)}+\cdots+\bar{n}_{j}^{(n)})\cdot v_j]}$ and $$\sigma_m:=(\frac{1}{d}\bar{n}_{j}^{(1)}\cdot v_j,\cdots,\frac{1}{d}\bar{n}_{j}^{(m-1)}\cdot v_j,\frac{1}{d}\bar{n}_{j}^{(m+1)}\cdot v_j,\cdots,\frac{1}{d}\bar{n}_{j}^{(n)}\cdot v_j).$$
For $f \in S$, we let $\mu_{\bar{k}_{m}}^{N,(\bar{k}_{m})}, \nu_{\bar{k}_{m}}^{N,(\bar{k}_{m})}$, and $\nu^{N,(\bar{k}_{m})}$ denote the measures satisfying
$$
\begin{aligned}
	\left\langle f, \mu_{\bar{k}_{m}}^{N,(\bar{k}_{m})}\right\rangle & = \int\left[f(c_N2^{-\sigma_m\cdot \bar{k}_{m}} \tilde{P}_j(\mathfrak{t}))-f(c_N2^{-\sigma_m\cdot \bar{k}_{m}} a_{v_{j}} \mathfrak{t}^{v_j})\right]\eta(t_1)\cdots \eta(t_n) d\mathfrak{t}  \\
	\left\langle f, \nu_{\bar{k}_{m}}^{N,(\bar{k}_{m})}\right\rangle & = \int\left[f(2^{-\sigma_m\cdot \bar{k}_{m}} \tilde{P}_j(\mathfrak{t}))-f(2^{-\sigma_m\cdot \bar{k}_{m}} a_{v_{j}} \mathfrak{t}^{v_j})\right]\eta(t_1)\cdots \eta(t_n) d\mathfrak{t} \\
	\left\langle f, \nu^{N,(\bar{k}_{m})}\right\rangle & = \int\left[f(\tilde{P}_j(\mathfrak{t}))-f(a_{v_{j}} \mathfrak{t}^{v_j})\right]\eta(t_1)\cdots \eta(t_n) d\mathfrak{t}
\end{aligned}
$$  for those $\bar{k}_{m}$ 's for which
$$
\mathfrak{q}=\frac{N+k_{1}}{d_j}\bar{n}_{j}^{(1)}+\cdots+\frac{N+k_{m-1}}{d_j}\bar{n}_{j}^{(m-1)}+\frac{N}{d_j}\bar{n}_{j}^{(m)}+\frac{N+k_{m+1}}{d_j}\bar{n}_{j}^{(m+1)}+\cdots+\frac{N+k_{n}}{d_j}\bar{n}_{j}^{(n)}.
$$
For all the other $\bar{k}_{m}$ 's, we define $\mu_{\bar{k}_{m}}^{N,(\bar{k}_{m})}, \nu_{\bar{k}_{m}}^{N,(\bar{k}_{m})}$, and $\nu^{N,(\bar{k}_{m})}$ to be zero distributions. Then, we define $
M_j^N f(x)$ as $$
M_j^N f(x):=\sup _{\bar{k}_{m} \in \mathbb{N}_{\star}^{n-1}}\left|\mu_{\bar{k}_{m}}^{N,(\bar{k}_{m})} * f(x)\right| .
$$
Also, we define
$$
L_j^N f(x):=\sup _{\bar{k}_{m} \in \mathbb{N}_{\star}^{n-1}}\left|\nu_{\bar{k}_{m}}^{N,(\bar{k}_{m})} * f(x)\right|.
$$Note that $(M_j^N f)(x)=(L_j^N f_N)\left(x / c_N\right)$ where $f_N(x)=f\left(c_N x\right)$. So, we shall consider the weak-type 1-1 estimate for $L_j^N$.\\
Now, we introduce the following Lemma whoose proof is postponed in the section 4.
\begin{lemma}\label{8a}	There are are positive real constants $A,B,C, \delta_1$ and $\delta_2$ independent of $\bar{k}_{m}$ and  $N$ such that \begin{align}\label{9a}
		\int\left|\nu^{N,(\bar{k}_{m})}(x-y)-\nu^{N,(\bar{k}_{m})}(x)\right| \mathrm{d} x \leq A 2^{-\delta_1 N}|y|^{\delta_2} \text { for all } y \in \mathbb{R},
	\end{align}
	\begin{align}\label{10a}
		\sum_{\bar{k}_{m} \in \mathbb{N}_{\star}^{n-1}} \int_{|x| \geq 2|y|}\left|\nu_{\bar{k}_{m}}^{N,(\bar{k}_{m})}(x-y)-\nu_{\bar{k}_{m}}^{N,(\bar{k}_{m})}(x)\right| \mathrm{d} x \leq C 2^{-\delta_1 N},
	\end{align}
\begin{align}\label{11a}
\left\|\sup _{\bar{k}_{m} \in \mathbb{N}^{n-1}}\left|f * \nu_{\bar{k}_{m}}^{N,(\bar{k}_{m})}\right|\right\|_2 \leq B 2^{-\delta_3 N}\|f\|_2 .
\end{align}
\end{lemma}
We shall apply Lemma \ref{8a} to show the following Proposition utilizing the Calderón-Zygmund singular integral theory. The following Proposition implies that $M_m(j)f(x)$ has the weak type 1-1 boundedness. 
\begin{proposition}\label{160a}
	\begin{align}\label{16a}
		\left|\left\{x: L_j^N f(x)>\lambda\right\}\right| \leq \frac{C}{\lambda} 2^{-\delta N}\|f\|_1
	\end{align}
	where $C$ and $\delta$ are positive real constants independent of $N$ and $\lambda$. 
\end{proposition}
	
\begin{proof}
First, we shall introduce the  the Calderón-Zygmund decomposition briefly.	We choose $\varepsilon>0$ (to be fixed later) and decompose $\mathbb{R}$ as $\mathbb{R}=F \cup \Omega$ so that:\\
	(1) $F$ is closed and $F \cap \Omega=\varnothing$;\\
	(2) $f(x) \leq 2^{\varepsilon N} \lambda$ a.e. on $F$;\\
	(3) $\Omega$ is the union of intervals, $\Omega=\bigcup_{i=1}^{\infty} Q_i$, whose interiors are mutually disjoint. Moreover,
\begin{align}
		|\Omega| \leq \frac{C}{2^{\varepsilon N} \lambda} \int_{\mathbb{R}} f(x) \mathrm{d} x,\label{12a}\\
		\frac{1}{\left|Q_i\right|} \int_{Q_i} f(x) \mathrm{d} x \leq C 2^{\varepsilon N} \lambda .\label{13a}
\end{align}The constant $C$ in (\ref{12a}) and (\ref{13a}) depends only on the dimension of the space. We define a function $g$ almost everywhere by
	\begin{align}\label{14a}
			g(x)= \begin{cases}f(x) & \text { for } x \in F \\ \frac{1}{\left|Q_i\right|} \int_{Q_i} f(y) \mathrm{d} y, & \text { for } x \in Q_i^{\circ}\end{cases}.
	\end{align}
	If $f(x)=g(x)+b(x)$, then $b(x)=0$ for $x \in F$ and
	\begin{align}\label{15a}
		\int_{Q_i} b(x) \mathrm{d} x=0, \text { for each } Q_i.
	\end{align}
	Also, one knows that
	$$
	|\{x: L_j^N f(x)>\lambda\}| \leq|\{x: L_j^N g(x)>\frac{\lambda}{2}\}|+|\{x: L_j^N b(x)>\frac{\lambda}{2}\}|.
	$$
By (\ref{12a}), (\ref{13a}) and (\ref{14a}), we have
	$$
	\begin{aligned}
		\|g\|_2^2 & =\int_F|g(x)|^2 \mathrm{~d} x+\int_{\Omega}|g(x)|^2 \mathrm{~d} x \leq C 2^{\varepsilon N} \lambda\|f\|_1 .
	\end{aligned}
	$$Therefore, by Lemma \ref{11a}, we obtain that
	$$
	\begin{aligned}
		\left|\left\{x: L_j^N g(x)>\frac{\lambda}{2}\right\}\right| & \leq \frac{4}{\lambda^2} \int\left|L_j^N g\right|^2 \mathrm{~d} x 
		=\frac{4}{\lambda^2} \int\left(\sup _{\bar{k}_{m}\in \mathbb{N}_{\star}^{n-1}}\left| \nu_{\bar{k}_{m}}^{N,(\bar{k}_{m})} * g(x)\right|\right)^2 \mathrm{~d} x \\
		& \leq \frac{4}{\lambda^2} B^2 2^{-2 \delta_3 N}\|g\|_2^2  \leq \frac{C}{\lambda} 2^{-\left(2 \delta_3-\varepsilon\right) N}\|f\|_1 .
	\end{aligned}
	$$ Set
	$$
	b_i(x)= \begin{cases}b(x) & \text { if } x \in Q_i \\ 0 & \text { if } x \notin Q_i.\end{cases}
	$$
	Then, one has
	$$
	\begin{gathered}
		b(x)=\sum_i b_i(x) \ \text{and} \ 
		L_j^N b(x) \leq \sum_i L_j^N b_i(x) .
	\end{gathered}
	$$
	Let $Q_i^*$ denote the interval which has the same centre $y^i$ as $Q_i$, but which is expanded 2 times. Then $Q_i \subseteq Q_i^*$ and $\Omega \subseteq \Omega^*$ where $\Omega^*=\cup Q_i^*$. Moreover, we have:\\
	(a) $\left|\Omega^*\right| \leq 2|\Omega|$ and $F^* \subseteq F$ where $F^*=\left(\Omega^*\right)^c$,\\
	(b) If $x \notin Q_i^*$, then $\left|x-y^i\right| \geq 2\left|y-y^i\right|$ for all $y \in Q_i$.\\
 Now, by (\ref{15a}) we have
	$$
	\begin{aligned}
		L_j^N b_i(x) & =\sup _{\bar{k}_{m}\in \mathbb{N}_{\star}^{n-1}}\left|\int \nu_{\bar{k}_{m}}^{N,(\bar{k}_{m})}(x-y) b_i(y) \mathrm{d} y\right| \\
		& =\sup _{\bar{k}_{m}\in \mathbb{N}_{\star}^{n-1}}\left|\int_{Q_i} \nu_{\bar{k}_{m}}^{N,(\bar{k}_{m})}(x-y) b(y) \mathrm{d} y\right| \\
		& =\sup _{\bar{k}_{m}\in \mathbb{N}_{\star}^{n-1}}\left|\int_{Q_i}\left[\nu_{\bar{k}_{m}}^{N,(\bar{k}_{m})}(x-y)-\nu_{\bar{k}_{m}}^{N,(\bar{k}_{m})}\left(x-y^i\right)\right] b(y) \mathrm{d} y\right|.
	\end{aligned}
	$$ So, we know that
	$$
	\begin{aligned}
		\int_{F^*} & L_j^N b(x) \mathrm{d}x \\
		& \leq \sum_i \int_{x \notin Q_i^*} L_j^N b_i(x) \mathrm{d} x \\
		& \leq \sum_i \int_{x \notin Q_i^*} \int_{y \in Q_i} \sup _{\bar{k}_{m}\in \mathbb{N}_{\star}^{n-1}}\left|\nu_{\bar{k}_{m}}^{N,(\bar{k}_{m})}(x-y)-\nu_{\bar{k}_{m}}^{N,(\bar{k}_{m})}\left(x-y^i\right)\right||b(y)| \mathrm{d} y \mathrm{~d} x \\
		& \leq \sum_i \int_{y \in Q_i}\left\{\int_{x \notin Q_i^*} \sum_{\bar{k}_{m}\in \mathbb{N}_{\star}^{n-1}}\left|\nu_{\bar{k}_{m}}^{N,(\bar{k}_{m})}(x-y)-\nu_{\bar{k}_{m}}^{N,(\bar{k}_{m})}\left(x-y^i\right)\right| \mathrm{d} x\right\}|b(y)| \mathrm{d} y \\
		& \leq \sum_i \int_{y \in Q_i}\left\{\int_{\left|x^{\prime}\right| \geq 2\left|y^{\prime}\right|} \sum_{\bar{k}_{m}\in \mathbb{N}_{\star}^{n-1}}\left|\nu_{\bar{k}_{m}}^{N,(\bar{k}_{m})}\left(x^{\prime}-y^{\prime}\right)-\nu_{\bar{k}_{m}}^{N,(\bar{k}_{m})}\left(x^{\prime}\right)\right| \mathrm{d} x^{\prime}\right\}|b(y)| \mathrm{d} y
	\end{aligned}
	$$
	using (b) for $y \in Q_i$ if $x^{\prime}=x-y^i$ and $y^{\prime}=y-y^i$. Then, by (\ref{10a}) we obtain that \begin{align*}
		\int_{F^*} L_j^N b(x) \mathrm{d} x & \leq C 2^{-\delta_1 N} \sum_i \int_{y \in Q_i}|b(y)| \mathrm{d} y \\
		& \leq C 2^{-\delta_1 N}\left\{\sum_i \int_{y \in Q_i}|f(y)| \mathrm{d} y+\sum_i \int_{y \in Q_i}|g(y)| \mathrm{d} y\right\} \\
		& \leq C 2^{-\delta_1 N}\left\{\int_{\Omega}|f(y)| \mathrm{d} y+\sum_i \int_{y \in Q_i}\left\{\frac{1}{\left|Q_i\right|} \int_{Q_i}|f(z)| \mathrm{d} z\right\} \mathrm{d} y\right\} \\
		& \leq C 2^{-\delta_1 N}\left\{\|f\|_1+\int_{\Omega}|f(z)| \mathrm{d} z\right\} \\
		& \leq C 2^{-\delta_1 N}\|f\|_1 .
	\end{align*} Thus, we have
$$
\begin{aligned}
\left|\left\{x: L_j^N b(x)>\frac{\lambda}{2}\right\}\right| & \leq\left|\left\{x \in F^*: L_j^N b(x)>\frac{\lambda}{2}\right\}\right|+\left|\left\{x \in \Omega^*: L_j^N b(x)>\frac{\lambda}{2}\right\}\right| \\
& \leq \frac{2}{\lambda} \int_{F^*} L_j^N b(x) \mathrm{d} x+\left|\Omega^*\right| \\
& \leq C\left(\frac{1}{\lambda} 2^{-\delta_1 N}\|f\|_1+\frac{2^{-\varepsilon N}}{\lambda}\|f\|_1\right) \\
& \leq \frac{C}{\lambda} 2^{-\delta_4 N}\|f\|_1
\end{aligned}
$$
\end{proof}
\subsection{Zero Coordinate Case} Now, we shall treat the case where $v_j=(m_{1}^{j},\cdots, m_{n}^{j})$ has a vanishing coordinate. Assume that \begin{align*}
	\begin{cases}&m_{i}^{j}=0  \text { for } i \in A, \\ &m_{i}^{j}\neq0  \text { for } i \in B \end{cases}
\end{align*}  where $v_j=(m_{1}^{j},\cdots, m_{n}^{j})$. Also, suppose that \begin{align}\label{19a}
	A=\{a_1,a_2,\cdots,a_{\alpha}\}\ \text{and} \ B=\{b_1,b_2,\cdots,b_{\beta}\} \quad (A\cup B=\{1,2\cdots,n\}).\end{align}Now, we set $$\Lambda_0:=\{v=(m_1,\cdots,m_n): m_i=0 \quad \text{if} \quad i\in A\}.$$
Then, we split $\mathcal{M}(j) f(x)$ as
\begin{align*}
	& \mathcal{M}(j) f(x)= \sup _{\mathfrak{q} \in S(j)}\left|\int f(x-P(2^{-q_1}t_1,\cdots, 2^{-q_n}t_n))\eta(t_1)\cdots \eta(t_n) \mathrm{d}t_1 \cdots\mathrm{d} t_n\right| \\
	& \leq \sup _{\mathfrak{q} \in S(j)} \left| \int[f(x-P(2^{-q_1}t_1,\cdots, 2^{-q_n}t_n))-f(x-2^{-\mathfrak{q} \cdot v_j}\sum_{v\in \Lambda_0}2^{-\mathfrak{q} \cdot (v-v_j)} a_{v} \mathfrak{t}^{v})]\eta(t_1)\cdots \eta(t_n) d\mathfrak{t} \right| \\
	&+\sup _{\mathfrak{q} \in S(j)}\left|\int f(x-2^{-\mathfrak{q} \cdot v_j}\sum_{v\in \Lambda_0}2^{-\mathfrak{q} \cdot (v-v_j)} a_{v} \mathfrak{t}^{v})\eta(t_1)\cdots \eta(t_n) d\mathfrak{t} \right| \\
	&:=M(j) f(x)+Q(j) f(x).
\end{align*}

First, we shall consider $Q(j)f(x)$. In order to show that $Q(j)f(x)$ is weak-type 1-1 with a bound dependent on the coefficients, we shall use the induction argument. Assume that $n_0$-parameter maximal function operator $\mathcal{M}'$ is weak-type 1-1 with a bound dependent on the coefficients of $P_0(t_1,\cdots,t_{n_0})$ where $n_0< n$ and \begin{align}\label{18a}
\mathcal{M}' f(x)=\sup _{0<h_1,h_2,\cdots,h_{n_0}<1} \frac{1}{h_1h_2\cdots h_{n_0}}\left|\int_0^{h_1}\cdots \int_0^{h_{n_0}} f(x-P_0(t_1,\cdots,t_{n_0})) \mathrm{d}t_1\cdots \mathrm{d} t_{n_0}\right|.
\end{align} Then, by the following inequality: \begin{align}
Q(j) f(x)&=\sup _{\mathfrak{q} \in S(j)}\left|\int f(x-2^{-\mathfrak{q} \cdot v_j}\sum_{v\in \Lambda_0}2^{-\mathfrak{q} \cdot (v-v_j)} a_{v} \mathfrak{t}^{v})\eta(t_1)\cdots \eta(t_n) d\mathfrak{t} \right|\notag\\
&=\sup _{\mathfrak{q} \in S(j)}\left|\int f(x-\sum_{v\in \Lambda_0}2^{-\mathfrak{q} \cdot v} a_{v} \mathfrak{t}^{v})\eta(t_1)\cdots \eta(t_n) d\mathfrak{t} \right|\notag\\
&\le \int\left\{\sup _{\mathfrak{q} \in S(j)}\left|\int f(x-\sum_{v\in \Lambda_0}2^{-\mathfrak{q} \cdot v} a_{v} \mathfrak{t}^{v})\prod_{i\in B}\eta(t_i) dt_{i\in B} \right|\right\}\prod_{i\in A}\eta(t_i)dt_{i\in A},\label{17a}
\end{align} we can obtain that $Q(j)f(x)$ is also weak-type 1-1 with a bound dependent on the coefficients of $P$ by the assumption. This is because (\ref{17a}) can be controlled by the $n_0$-parameter maximal function that is mentioned in (\ref{18a}).\\
Now, we shall treat $M(j) f(x)$. We write $P\left(2^{-q_1} t_1,\cdots, 2^{-q_n} t_n\right)=2^{-\mathfrak{q} \cdot v_j} \tilde{P}(\mathfrak{t})$ where
$$
\tilde{P}(\mathfrak{t})=\sum_{v=\mathfrak{m} \in \Lambda} 2^{-\mathfrak{q} \cdot\left(v-v_j\right)} a_{\mathfrak{m}} \mathfrak{t}^{\mathfrak{m}}
$$
and set
$$
\tilde{P}_0(\mathfrak{t})=\sum_{v=\mathfrak{m} \in  \Lambda_0} 2^{-\mathfrak{q} \cdot\left(v-v_j\right)} a_{\mathfrak{m}} t^{\mathfrak{m}} .
$$Hence, we rewrite $M(j)f(x)$ as \begin{align}\label{20a}
	& M(j) f(x)\notag\\
	& =\sup _{\mathfrak{q} \in S(j)} | \int\left[f(x-2^{-\mathfrak{q} \cdot v_j} \tilde{P}(\mathfrak{t}))-f(x-2^{-\mathfrak{q} \cdot v_j} \tilde{P}_0(\mathfrak{t}))\right] \eta(t_1)\cdots \eta(t_n) d\mathfrak{t}.
\end{align}Note that for $v_j=(m_{1}^{j},\cdots, m_{n}^{j})$ and each $i\in A$ in (\ref{19a}), there exists non-negative vector $n_{j}^{(a_i)}$ which is parallel with the vector $i$-th axis. Then, let $n_{j}^{(b_1)},\cdots,n_{j}^{(b_{\beta})}$ denote the other normal non-negative vectors generated by the corner point $v_j$.\\
For $\bar{k}=(k_1,k_2,\cdots,k_{\alpha})\in \mathbb{N}_{\star}^{\alpha}$, we define $S^{\bar{k}}(j)$ as $$S^{\bar{k} }(j)=\left\{\mathfrak{q} \in S(j):\mathfrak{q}=\frac{k_1}{d} n_{j}^{(a_1)}+\cdots+\frac{k_{\alpha}}{d}n_{j}^{(a_{\alpha})}+\frac{\ell_1}{d} n_{j}^{(b_1)}+\cdots  +\frac{\ell_{\beta}}{d} n_{j}^{(b_{\beta})}; \bar{\ell}=(\ell_1,\cdots,\ell_{\beta}) \in \mathbb{N}_{\star}^{\beta}\right\}.$$ Then, (\ref{20a}) is bounded by \begin{align*}
	&\sum_{\bar{k}\in\mathbb{N}_{\star}^{\alpha}} \sup _{\mathfrak{q} \in S^{\bar{k}}(j)} \left| \int\left[f(x-2^{-\mathfrak{q} \cdot v_j} \tilde{P}(\mathfrak{t}))-f(x-2^{-\mathfrak{q} \cdot v_j} \tilde{P}_0(\mathfrak{t}))\right] \eta(t_1)\cdots \eta(t_n) d\mathfrak{t}\right|\\
	&:=\sum_{\bar{k}\in\mathbb{N}_{\star}^{\alpha}}M^{\bar{k}}(j)f(x).
\end{align*}
So, if we have the following Proposition, then one can obtain the weak type 1-1 bound of the operator $M(j)$.
\begin{proposition}There exists $\gamma'\in \mathbb{Q}_{+}^{\beta}$ such that
	\begin{align}\label{2d}
		\left|\left\{x: M^{\bar{k}}(j) f(x)>\lambda\right\}\right| \leq \frac{C}{\lambda} 2^{-\gamma' \cdot \bar{k}}\|f\|_1 
	\end{align}
	where $\gamma'$ and $C>0$ are independent of $\bar{k}$ and $\lambda$.
\end{proposition}
\begin{proof} 
	Observe that for $\mathfrak{q} \in S^k(j)$,
	$$
	\begin{aligned}
		\mathfrak{q} \cdot v_j & =(\frac{\ell_1}{d} n_{j}^{(b_1)}+\cdots  +\frac{\ell_{\beta}}{d} n_{j}^{(b_{\beta})})\cdot v_{j}
	\end{aligned}
	$$
	since $\bar{n}_{j}^{(a_i)} \cdot v_j=0$ for $i\in A$. We set
	$$
	\sigma=\left(\frac{1}{d}(\bar{n}_{j}^{(b_1)} \cdot v_j),\cdots, \frac{1}{d}(\bar{n}_{j}^{(b_{\beta})} \cdot v_j)\right).
	$$ One can easily check that $\sigma$ is not a zero vector. Then, we rewrite $M^{\bar{k}}(j) f(x)$ with $\bar{\ell}=(\ell_1,\cdots,\ell_{\beta})$  as $$
	\begin{aligned}
		M^{\bar{k}}(j) f(x)= & \sup _{\bar{\ell} \in\mathbb{N}_{\star}^{\beta}} \left|\int\left[f(x-2^{-\sigma \cdot \bar{\ell}} \tilde{P}(\mathfrak{t}))-f(x-2^{-\sigma \cdot \bar{\ell}} \tilde{P}_0(\mathfrak{t}))\right] \eta(t_1)\cdots \eta(t_n) d\mathfrak{t}\right|\\
		&:=\sup _{\bar{\ell} \in\mathbb{N}_{\star}^{\beta}}\left| \nu_{\bar{\ell}}^{\bar{k},(\bar{\ell})} * f(x)\right|
	\end{aligned}
	$$
	where $v_{\bar{\ell}}^{\bar{k},(\bar{\ell})}$ and $v^{\bar{k},(\bar{\ell})}$ are the measures defined by
	$$
	\begin{aligned}
		\left\langle f, v_{\bar{\ell}}^{\bar{k},(\bar{\ell})}\right\rangle & =\int\left[f(2^{-\sigma \cdot \bar{\ell}} \tilde{P}(\mathfrak{t}))-f(2^{-\sigma\cdot \bar{\ell}} \tilde{P}_0(\mathfrak{t}))\right]\eta(t_1)\cdots \eta(t_n) d\mathfrak{t} \\
		\left\langle f, v^{\bar{k},(\bar{\ell})}\right\rangle & =\int\left[f(\tilde{P}(\mathfrak{t}))-f(\tilde{P}_0(\mathfrak{t}))\right] \eta(t_1)\cdots \eta(t_n) d\mathfrak{t}
	\end{aligned}
	$$
	for those $\bar{\ell}$ 's for which
	$$
	\frac{k_1}{d} n_{j}^{(a_1)}+\cdots+\frac{k_{\alpha}}{d}n_{j}^{(a_{\alpha})}+\frac{\ell_1}{d} n_{j}^{(b_1)}+\cdots  +\frac{\ell_{\beta}}{d} n_{j}^{(b_{\beta})}= \mathfrak{q}\in S^{\bar{k}}(j)
	$$
	and for all the other $\bar{\ell}$ 's we define them to be zero distributions.
	So, the proof is same as proof of Proposition \ref{160a} if we have the following lemmas.
\end{proof}

\begin{lemma}\label{21a}There are positive real constant $C,\delta$ and  vector $\gamma\in \mathbb{Q}_{+}^{\beta}$ independent of $\bar{k}$ and  $\bar{\ell}$ such that \begin{align}\label{22a}
		\int\left|v^{\bar{k},(\bar{\ell})}(x-y)-v^{\bar{k},(\bar{\ell})}(x)\right| \mathrm{d} x \leq C 2^{-\gamma\cdot \bar{k}}|y|^{\delta} \text { for all } y \in \mathbb{R},
	\end{align}
	\begin{align}\label{23a}
		\sum_{\bar{\ell} \in \mathbb{N}_{\star}^{\beta}} \int_{|x| \geq 2|y|}\left|v_{\bar{\ell}}^{\bar{k},(\bar{\ell})}(x-y)-v_{\bar{\ell}}^{\bar{k},(\bar{\ell})}(x)\right| \mathrm{d} x \leq C 2^{-\gamma\cdot k},
	\end{align}
	\begin{align}\label{24a}
		\left\|\sup _{\bar{\ell} \in \mathbb{N}_{\star}^{\beta}}\left|f * v_{\bar{\ell}}^{\bar{k},(\bar{\ell})}\right|\right\|_2 \leq C 2^{-\gamma\cdot k}\|f\|_2 .
	\end{align}
\end{lemma}

\section{Proof Of Lemma \ref{8a}}
In this section, we shall show  Lemma \ref{8a}.
\subsection{Proof of (\ref{9a})} Recall that we write $P(2^{-q_1}t_1,\cdots, 2^{-q_n}t_n)=2^{-\mathfrak{q}\cdot v_j}\tilde{P}_j(t_1,\cdots,t_n)$ where
$$
\tilde{P}_j(\mathfrak{t}):=\sum_{v=\mathfrak{m} \in \Lambda} 2^{-\mathfrak{q} \cdot\left(v-v_j\right)} a_{\mathfrak{m}} \mathfrak{t}^{\mathfrak{m}}=a_{v_j}\mathfrak{t}^{v_j}+\sum_{v=\mathfrak{m} \in \Lambda\setminus{v_j}} 2^{-\mathfrak{q} \cdot\left(v-v_j\right)} a_{\mathfrak{m}} \mathfrak{t}^{\mathfrak{m}}.
$$ Since $\mathfrak{q}\cdot (v-v_j) \ge 0$ for any $\mathfrak{q}\in S(j)$, $\tilde{P}_j(\mathfrak{t})$ has uniformly bounded $\mathcal{C}_{m}$ norm for all $\mathfrak{q}\in S(j)$. Moreover, by the term $a_{v_j}\mathfrak{t}^{v_j}$ of   $\tilde{P}_j(\mathfrak{t})$, one can knows that $\partial^\alpha \tilde{P}_j(\mathfrak{t})$ is uniformly bounded below for some $\alpha$ with $1 \leq|\alpha| \leq \operatorname{deg} \tilde{P}_j$. $\tilde{P}_0(\mathfrak{t})$ also have two same properties. Hence, by Proposition 7.2 of \cite{4} combined with the above two properties of $\tilde{P}_j(\mathfrak{t})$ and $\tilde{P}_0(\mathfrak{t})$, we have \begin{align}\label{5d}
\int\left|\nu^{N,(\bar{k}_{m})}(x-y)-\nu^{N,(\bar{k}_{m})}(x)\right| \mathrm{d} x \leq C |y|^{\epsilon_0}
\end{align}for all $y \in \mathbb{R}$ with $C,\epsilon_0>0$ independent of $\bar{k}_m$ and $N$.

Now, we consider the case where $N$ is sufficiently large. By Lemma \ref{1b}, when $N$ is sufficiently large,  $\nabla \tilde{P}_j(\mathfrak{t})$ is uniformly bounded below within the support of $ \eta(t_1)\cdots \eta(t_n)$ . Hence, one has \begin{align}\label{5c}
	& \left|\widehat{ \nu^{N,(\bar{k}_{m})}}(\xi)\right|=\left|\int\left[\exp (i \xi \tilde{P}_j(\mathfrak{t}))-\exp \left(i \xi a_{v_j}  \mathfrak{t}^{v_j}\right)\right]\eta(t_1)\cdots \eta(t_n) d\mathfrak{t}\right| \leq \frac{C_m}{|\xi|^m}
\end{align} for all  $m \in \mathbb{N}$ and  $N \geq N_0$. Here, $N_0$ depends only on the coefficients of polynomial $P$.  Also, by the mean value theorem, we have \begin{align}\label{6c}
\left|\widehat{ \nu^{N,(\bar{k}_{m})}}(\xi)\right|\le C2^{-\beta_jN}|\xi|
\end{align}  where $\beta_j$ is defined in (\ref{2b}). Thus, by (\ref{5c}) and (\ref{6c}) combined with the convexity, there exists $\sigma_j>0$ such that \begin{align}
& \int\left| \nu^{N,(\bar{k}_{m})}(x-y)- \nu^{N,(\bar{k}_{m})}(x)\right| \mathrm{d} x \notag\\
&\le C\left\{\int\left|(e^{-2 \pi i y \xi}-1) \cdot \widehat{\nu^{N,(\bar{k}_{m})}}(\xi)\right|^2 \mathrm{~d} \xi\right\}^{1 / 2} \label{7c}\\
&=C\left\{\int_{|\xi| \leq 1}\left|(e^{-2 \pi i y \xi}-1) \cdot \widehat{\nu^{N,(\bar{k}_{m})}}(\xi)\right|^2 \mathrm{~d} \xi+\int_{|\xi| \geq 1}\left|(e^{-2 \pi i y \xi}-1) \cdot \widehat{\nu^{N,(\bar{k}_{m})}}(\xi)\right|^2 \mathrm{~d} \xi\right\}^{1 / 2}\notag \\
& \leq C 2^{-\sigma_j N}|y| .\label{8c}
\end{align} Here, we have used the Cauchy-Schwarz inequality with  the fact that $ v^{N,(\bar{k}_{m})}$ has a compact support and Plancherel's theorem to have (\ref{7c}).
\subsection{Proof of (\ref{10a})}
 Suppose that the uniform compact support of the $\nu^{N,(\bar{k}_{m})}$ (for all $N$ and $\bar{k}_{m}$) is contained in a ball $B(0, R)$. Since
$$
\nu_{\bar{k}_{m}}^{N,(\bar{k}_{m})}(x)=2^{\sigma_m\cdot \bar{k}_{m}} \nu^{N,(\bar{k}_{m})}\left(2^{\sigma_m\cdot \bar{k}_{m}} x\right),
$$
the support of $v_k^{N,(k)}$ is contained in $B(0, 2^{-\sigma_m\cdot \bar{k}_{m}} R)$. Let  $2^{k_0-1}<|y| \leq 2^{k_0}$ for some $k_0 \in \mathbb{Z}$. Then, $|x| \geq 2|y|$ means that
$$
|x|>2^{k_0} \text { and }|x-y| \geq|y|>2^{k_0-1}.
$$ Hence, if $2^{-\sigma_m\cdot \bar{k}_{m}} R < 2^{k_0-1}$, $$\int_{|x| \geq 2|y|}\left|\nu_{\bar{k}_{m}}^{N,(\bar{k}_{m})}(x-y)-\nu_{\bar{k}_{m}}^{N,(\bar{k}_{m})}(x)\right| \mathrm{d} x \equiv 0 .$$
Therefore, by (\ref{9a}) we have$$
\begin{aligned}
	& \sum_{\bar{k}_{m}} \int_{|x| \geq 2|y|}\left|\nu_{\bar{k}_{m}}^{N,(\bar{k}_{m})}(x-y)-\nu_{\bar{k}_{m}}^{N,(\bar{k}_{m})}(x)\right| \mathrm{d} x \\
	& \leq \sum_{\bar{k}_{m}: \sigma_m\cdot\bar{k}_{m}+k_0 \leq \log R+1} \int\left|\nu_{\bar{k}_{m}}^{N,(\bar{k}_{m})}(x-y)-\nu_{\bar{k}_{m}}^{N,(\bar{k}_{m})}(x)\right| \mathrm{d} x \\
	&=\sum_{\bar{k}_{m}: \sigma_m\cdot \bar{k}_{m}+k_0 \leq \log R+1} \int\left|\nu^{N,(\bar{k}_{m})}\left(z-2^{\sigma_m \cdot \bar{k}_{m}} y\right)-\nu^{N,(\bar{k}_{m})}(z)\right| \mathrm{d} z \\
	& \leq \sum_{\bar{k}_{m}: \sigma_m\cdot\bar{k}_{m}+k_0 \leq \log R+1} A 2^{-\delta_1 N}\left(2^{\sigma_m \cdot \bar{k}_{m}}|y|\right)^{\delta_2} \\
	& \leq A 2^{-\delta_1 N} \sum_{\bar{k}_{m}: \sigma_m\cdot\bar{k}_{m}+k_0 \leq \log R+1}\left(2^{\sigma_m \cdot \bar{k}_{m}} 2^{k_0}\right)^{\delta_2} \\
	& \leq A 2^{-\delta_1 N} C\left(R, \sigma_1, \delta_2\right).
\end{aligned}
$$
\subsection{Proof of (\ref{11a})}
By the mean value theorem, one has
\begin{align}\label{10c}
	&\left|\widehat{ \nu_{\bar{k}_{m}}^{N,(\bar{k}_{m})}}(\xi)\right| \leq C|\xi| 2^{-\sigma_m \cdot\bar{k}_{m}} 2^{-\beta_j N}.
\end{align}
Furthermore, by the van der Corput’s lemma, one knows that
\begin{align}\label{11c}
		&\left|\widehat{ \nu_{\bar{k}_{m}}^{N,(\bar{k}_{m})}}(\xi)\right| \leq \frac{C}{\left(|\xi| 2^{-\sigma_m \cdot\bar{k}_{m}}\right)^{\epsilon}}.
\end{align}
Therefore, by the convexity, there exists $\delta_3>0$ such that
\begin{align}\label{12c}
	&\left|\widehat{ \nu_{\bar{k}_{m}}^{N,(\bar{k}_{m})}}(\xi)\right| \leq C2^{-\delta_3N}\min\left(|\xi|2^{-\sigma_m \cdot\bar{k}_{m}},\frac{1}{\left(|\xi| 2^{-\sigma_m \cdot\bar{k}_{m}}\right)^{\epsilon/2}}\right).
\end{align}
Thus, by the standard Littlewood-Paley theory, we have 	\begin{align}
	\left\|\sup _{\bar{k}_{m} \in \mathbb{N}_{\star}^{n-1}}\left|f * \nu_{\bar{k}_{m}}^{N,(\bar{k}_{m})}\right|\right\|_2 \leq B 2^{-\delta_3 N}\|f\|_2 .
\end{align}
For the details, see \cite{7}.

\section{Proof Of Lemma \ref{21a}}
The proof of (\ref{23a}) and (\ref{24a}) in Lemma \ref{21a} are same as the proof of (\ref{10a}) and (\ref{11a}) respectively. Hence, we shall show the inequality (\ref{22a}).
	
\subsection{Proof of (\ref{22a})}
	Recall that we write $P\left(2^{-q_1} t_1,\cdots, 2^{-q_n} t_n\right)=2^{-\mathfrak{q} \cdot v_j} \tilde{P}(\mathfrak{t})$ where
	$$
	\tilde{P}(\mathfrak{t})=\sum_{v=\mathfrak{m} \in \Lambda} 2^{-\mathfrak{q} \cdot\left(v-v_j\right)} a_{\mathfrak{m}} \mathfrak{t}^{\mathfrak{m}}
	$$
	and
	$$
	\tilde{P}_0(\mathfrak{t})=\sum_{v=\mathfrak{m} \in  \Lambda_0} 2^{-\mathfrak{q} \cdot\left(v-v_j\right)} a_{\mathfrak{m}} t^{\mathfrak{m}} .$$

Observe that for $v \in \Lambda \backslash \Lambda_0$ and $i\in A$, $\bar{n}_{j}^{(a_i)} \cdot\left(v-v_j\right)>0$  from the definition of $\Lambda_0$. So, we set $\gamma_{j}=(\gamma_{j}^{1},\cdots,\gamma_{j}^{\alpha})$ where
$$
\gamma_{j}^{i}=\frac{1}{d} \min _{v \in \Lambda \backslash \Lambda_0}\left\{\bar{n}_{j}^{(a_i)} \cdot\left(v-v_j\right)\right\}>0 , \ (i\in \{1,2,\cdots,\alpha\}).
$$ 
Now, we rewrite $\tilde{P}(\mathfrak{t})=\tilde{P}_0(\mathfrak{t})+2^{-\gamma_{j}\cdot \bar{k}} \tilde{Q}(\mathfrak{t})$ where $\bar{k}=(k_1,k_2,\cdots,k_{\alpha})\in \mathbb{N}_{\star}^{\alpha}$.
By the fact that for $v \in \Lambda \backslash \Lambda_0$ and $\mathfrak{q} \in S(j)$,
\begin{align}
	\mathfrak{q} \cdot\left(v-v_1\right) & \geq \frac{k_1}{d} \bar{n}_{j}^{(a_1)} \cdot\left(v-v_j\right)+\cdots + \frac{k_{\alpha}}{d} \bar{n}_{j}^{(a_{\alpha})} \cdot\left(v-v_j\right)\notag\\
	& \geq \gamma_{j}\cdot \bar{k}\label{19c},
\end{align}
one knows that there exists a constant $M>0$(depends only on the coefficients of P) such that \begin{align}\label{14c}
	\left\|\nabla\tilde{Q}(\mathfrak{t})\right\|_{L^{\infty}([1 / 2,4]^{n})}\le M.
\end{align}
Let
$$
I_{\theta}:=\{\mathfrak{t} \in[\frac{1}{2}, 4]^{n}:|\nabla\tilde{P}_0(t)|>2^{-\theta (\gamma_{j}\cdot  \bar{k})}\}.
$$ for some sufficiently small $\theta\in \mathbb{Q}_{+}$. For $I_{\theta}$, we set $$\widehat{ v_{I_{\theta}}^{\bar{k},(\bar{\ell})}}(\xi):=\int_{I_{\theta}} \left[e^{-2\pi i \xi\cdot \tilde{P}(\mathfrak{t})}-e^{-2\pi i \xi\cdot \tilde{P}_{0}(\mathfrak{t})}\right] \eta(t_1)\cdots \eta(t_n) d\mathfrak{t}$$ for those $\bar{\ell}$ 's for which
$$
\frac{k_1}{d} n_{j}^{(a_1)}+\cdots+\frac{k_{\alpha}}{d}n_{j}^{(a_{\alpha})}+\frac{\ell_1}{d} n_{j}^{(b_1)}+\cdots  +\frac{\ell_{\beta}}{d} n_{j}^{(b_{\beta})}= \mathfrak{q}\in S^{\bar{k}}(j)
$$
and for all the other $\bar{\ell}$ 's we define the Fourier multiplier to be zero.
Then, by the inequality (\ref{14c}) combined with $\tilde{P}(\mathfrak{t})=\tilde{P}_0(\mathfrak{t})+2^{-\gamma_{j}\cdot \bar{k}} \tilde{Q}(\mathfrak{t})$, one can obtain that 
\begin{align}\label{13c}
\left|\widehat{ v_{I_{\theta}}^{\bar{k},(\bar{\ell})}}(\xi)\right|\le \frac{C_{m}}{|\xi|^{m}}
\end{align} for all $m\in \mathbb{N}$ and $\bar{k}\ge (k_0,k_0\cdots,k_0)$ when we fix $\theta$ such that $\theta \le \theta_0$. Here, $\theta_0$ is determined by the $M$ in (\ref{14c}) and $k_0$ is choosen by the $M$ in (\ref{14c}) and the sufficiently small $\theta$ satisfying $\theta\le \theta_0$.\\
Furthermore, by the mean-value theorem with (\ref{19c}), we have \begin{align}\label{20c}
	\left|\widehat{ v_{I_{\theta}}^{\bar{k},(\bar{\ell})}}(\xi)\right|\le C|\xi|2^{-\gamma_{j}\cdot \bar{k}}.
\end{align} Therefore, by the same argument in the proof of (\ref{9a}) combined with (\ref{13c}) and (\ref{20c}), we have  \begin{align}\label{9d}
\int\left|v_{I_{\theta}}^{\bar{k},(\bar{\ell})}(x-y)-v_{I_{\theta}}^{\bar{k},(\bar{\ell})}(x)\right| \mathrm{d} x \leq C 2^{-\gamma \cdot \bar{k}}|y|^{\delta} 
\end{align} for all $y \in \mathbb{R}$ and $\bar{k}\ge (k_0,k_0\cdots,k_0)$.
On the other hand, using Proposition 7.2 of \cite{4}, we also know that \begin{align}\label{6d}
	\int\left|v_{I_{\theta}}^{\bar{k},(\bar{\ell})}(x-y)-v_{I_{\theta}}^{\bar{k},(\bar{\ell})}(x)\right| \mathrm{d} x \leq C |y|^{\epsilon_0}
\end{align}
for all $y \in \mathbb{R}$ with $C,\epsilon_0>0$ independent of $\bar{k}$, $\bar{\ell}$  and $I_{\theta}$. For the details, see the arguments to have ($\ref{5d}$).

Now, we shall consider $\widehat{ v_{I_{\theta}^{c}}^{\bar{k},(\bar{\ell})}}(\xi)$ defined by $$\widehat{ v_{I_{\theta}^{c}}^{\bar{k},(\bar{\ell})}}(\xi):=\widehat{ v^{\bar{k},(\bar{\ell})}}(\xi)-\widehat{ v_{I_{\theta}}^{\bar{k},(\bar{\ell})}}(\xi)$$  for those $\bar{\ell}$ 's for which
$$
\frac{k_1}{d} n_{j}^{(a_1)}+\cdots+\frac{k_{\alpha}}{d}n_{j}^{(a_{\alpha})}+\frac{\ell_1}{d} n_{j}^{(b_1)}+\cdots  +\frac{\ell_{\beta}}{d} n_{j}^{(b_{\beta})}= \mathfrak{q}\in S^{\bar{k}}(j)
$$
For all the other $\bar{\ell}$ 's, we define the Fourier multiplier to be zero. By the sublevel set estiamte, there exists $\delta'>0$ such that \begin{align}\label{15c}
I_{\theta}^{c}:=	\left|\{\mathfrak{t} \in[\frac{1}{2}, 4]^{n}:|\nabla\tilde{P}_0(t)|\le2^{-\theta (\gamma_{j}\cdot  \bar{k})}\}\right|\le C(2^{-\theta (\gamma_{j}\cdot  \bar{k})})^{\delta'}.
\end{align}So, due to (\ref{15c}), we have
\begin{align}\label{71c}
	\int\left|v_{I_{\theta}^{c}}^{\bar{k},(\bar{\ell})}(x)\right| \mathrm{d} x &\leq \int |\int_{I_{\theta}^{c}}\int e^{2\pi(x-\tilde{P}(\mathfrak{t}))\xi}d\xi d\mathfrak{t}|dx+ \int |\int_{I_{\theta}^{c}}\int e^{2\pi(x-\tilde{P}_{0}(\mathfrak{t}))\xi}d\xi d\mathfrak{t}|dx\\
	&= |I_{\theta}^{c}|(\int |\delta(x)|dx+ \int |\delta(x)|dx) \notag\\
	&\le C(2^{-\theta (\gamma_{j}\cdot  k)})^{\delta'}\notag
\end{align}where $\delta(x)$ is the Dirac delta distribution.
On the other hand, by Proposition 7.2 of \cite{4} again, one can obtain that
\begin{align}\label{70c}
	\int\left|v_{I_{\theta}^{c}}^{\bar{k},(\bar{\ell})}(x-y)-v_{I_{\theta}^{c}}^{\bar{k},(\bar{\ell})}(x)\right| \mathrm{d} x \leq C |y|^{\epsilon_0}
\end{align}
for all $y \in \mathbb{R}$ with $C,\epsilon_0>0$ independent of $\bar{k}$ and $\bar{\ell}$.\\
Therefore, by (\ref{71c}) and (\ref{70c}) we have 
\begin{align}\label{10d}
	\int\left|v_{I_{\theta}^{c}}^{\bar{k},(\bar{\ell})}(x-y)-v_{I_{\theta}^{c}}^{\bar{k},(\bar{\ell})}(x)\right| \mathrm{d} x \leq C 2^{-\gamma \cdot \bar{k}}|y|^{\delta} \text { for all } y \in \mathbb{R}.
\end{align}
(\ref{9d}), (\ref{6d}) and (\ref{10d}) implies the desired results.

\end{document}